\documentclass[reqno]{amsart}
\usepackage{amsmath,amssymb,amsthm,cases,mathtools}
\usepackage{type1cm}
\usepackage[dvipdfm]{hyperref}
\usepackage{array}

\usepackage{amscd}
\usepackage[all]{xy}
\usepackage{enumerate}

\usepackage[dvips]{graphicx}
\usepackage{color}

  \makeatletter
  \@addtoreset{equation}{section}
  \makeatother

\theoremstyle{definition}
\newtheorem{dfn}{Definition}[section]
\newtheorem{rem}[dfn]{Remark}

\newtheorem{que}[dfn]{Question}
\theoremstyle{plain}
\newtheorem{thm}[dfn]{Theorem}

\newtheorem{prp}[dfn]{Proposition}

\DeclareMathOperator*{\bigast}{\raisebox{-0.8ex}{\scalebox{2.5}{$\ast$}}}

\newcommand{\Aut}{\operatorname{Aut}}
\newcommand{\Bir}{\operatorname{Bir}}
\newcommand{\PGL}{\operatorname{PGL}}
\newcommand{\Dec}{\operatorname{Dec}}
\newcommand{\Ine}{\operatorname{Ine}}
\newcommand{\id}{\operatorname{id}}
\newcommand{\Cr}{\operatorname{Cr}}
\newcommand{\Pic}{\operatorname{Pic}}

\newcommand{\UC}{\operatorname{UC}}

\newcommand{\Exc}{\operatorname{Exc}}
\newcommand{\discrep}{\operatorname{discrep}}

\newcommand{\codim}{\operatorname{codim}}

\newcommand{\Ind}{\operatorname{Ind}}

\newcommand{\A}{\mathbb A}
\newcommand{\C}{\mathbb C}
\newcommand{\Z}{\mathbb Z}

\newcommand{\Q}{\mathbb Q}
\newcommand{\R}{\mathbb R}
\newcommand{\PP}{\mathbb P}

\newcommand{\OO}{\mathcal O}

\newcommand{\cc}{\colon}
\newcommand{\ceq}{\coloneqq}

\title[Calabi-Yau hypersurfaces and inertia groups]{Calabi-Yau hypersurfaces in the direct product of $\PP^{1}$ and inertia groups}
\author[Masakatsu Hayashi and Taro Hayashi]{Masakatsu Hayashi and Taro Hayashi}
\address{
(Masakatsu Hayashi)
Department of Mathematics,
Graduate School of Science,
Osaka University,
Machikaneyamacho 1-1, Toyonaka, Osaka 560-0043, Japan
}
\email{m-hayashi@cr.math.sci.osaka-u.ac.jp}
\address{
(Taro Hayashi)
Department of Mathematics,
Graduate School of Science,
Osaka University,
Machikaneyamacho 1-1, Toyonaka, Osaka 560-0043, Japan
}
\email{tarou-hayashi@cr.math.sci.osaka-u.ac.jp}
\date{\today}

\begin{document}

\maketitle


\begin{abstract}
We produce the family of Calabi-Yau hypersurfaces $X_{n}$ of $(\PP^{1})^{n+1}$ in higher dimension whose inertia group contains non commutative free groups.
This is completely different from Takahashi's result \cite{ta98} for Calabi-Yau hypersurfaces $M_{n}$ of $\PP^{n+1}$. 
\end{abstract}


\section{Introduction}

Throughout this paper, we work over $\C$.
Given an algebraic variety $X$, it is natural to consider its birational automorphisms $\varphi \cc X \dashrightarrow X$.
The set of these birational automorphisms forms a group $\Bir(X)$ with respect to the composition.
When $X$ is a projective space $\PP^{n}$ or equivalently an $n$-dimensional rational variety, this group is called the Cremona group.
In higher dimensional case ($n \geq 3$), though many elements of the Cremona group have been described, its whole structure is little known.

Let $V$ be an $(n+1)$-dimensional smooth projective rational manifold.
In this paper, we treat subgroups called the ``inertia group" (defined below \eqref{inertia}) of some hypersurface $X \subset V$ originated in \cite{gi94}.
It consists of those elements of the Cremona group that act on $X$ as identity.

In Section \ref{cyn}, we mention the result (Theorem \ref{tak}) of Takahashi \cite{ta98} about the smooth Calabi-Yau hypersurfaces $M_{n}$ of $\PP^{n+1}$ of degree $n+2$ (that is, $M_{n}$ is a hypersurface such that it is simply connected, there is no holomorphic $k$-form on $M_{n}$ for $0<k<n$, and there is a nowhere vanishing holomorphic $n$-form $\omega_{M_{n}}$).
It turns out that the inertia group of $M_{n}$ is trivial (Theorem \ref{intro2}).
Takahashi's result (Theorem \ref{tak}) is proved by using the ``Noether-Fano inequality".
It is the useful result that tells us when two Mori fiber spaces are isomorphic.
Theorem \ref{intro2} is a direct consequence of Takahashi's result.

In Section \ref{cy1n}, we consider Calabi-Yau hypersurfaces
\[
X_{n} = (2, 2, \ldots , 2) \subset (\PP^{1})^{n+1}.
\]
Let
\[
\UC(N) \ceq \overbrace{\Z/2\Z * \Z/2\Z * \cdots * \Z/2\Z}^{N} = \bigast_{i=1}^{N}\langle t_{i}\rangle
\]
be the \textit{universal Coxeter group} of rank $N$ where $\Z/2\Z$ is the cyclic group of order 2.
There is no non-trivial relation between its $N$ natural generators $t_{i}$.
Let
\[
p_{i} \cc X_{n} \to (\PP^{1})^{n}\ \ \ (i=1, \ldots , n+1)
\]
be the natural projections which are obtained by forgetting the $i$-th factor of $(\PP^{1})^{n+1}$.
Then, the $n+1$ projections $p_{i}$ are generically finite morphism of degree 2.
Thus, for each index $i$, there is a birational transformation
\[
\iota_{i} \cc X_{n} \dashrightarrow X_{n}
\]
that permutes the two points of general fibers of $p_{i}$ and this provides a group homomorphism
\[
\Phi \cc \UC(n+1) \to \Bir(X_{n}).
\]

From now, we set $P(n+1) \ceq (\PP^{1})^{n+1}$.
Cantat-Oguiso proved the following theorem in \cite{co11}.
\begin{thm}$($\cite[Theorem 1.3 (2)]{co11}$)$\label{iota}
Let $X_{n}$ be a generic hypersurface of multidegree $(2,2,\ldots,2)$ in $P(n+1)$ with $n \geq 3$.
Then the morphism $\Phi$ that maps each generator $t_{j}$ of $\UC(n+1)$ to the involution $\iota_{j}$ of $X_{n}$ is an isomorphism from $\UC(n+1)$ to $\Bir(X_{n})$.
\end{thm}
Here ``generic'' means $X_{n}$ belongs to the complement of some countable union of proper closed subvarieties of the complete linear system $\big| (2, 2, \ldots , 2)\big|$.

Let $X \subset V$ be a projective variety.
The \textit{decomposition group} of $X$ is the group
\begin{align*}
\Dec(V, X) \ceq \{f \in \Bir(V)\ |\ f(X) =X \text{ and } f|_{X} \in \Bir(X) \}.
\end{align*}
The \textit{inertia group} of $X$ is the group
\begin{align}\label{inertia}
\Ine(V, X) \ceq \{f \in \Dec(V, X)\ |\ f|_{X} = \id_{X}\}.
\end{align}

Then it is natural to consider the following question:

\begin{que}\label{qu}
Is the sequence
\begin{align}\label{se}
1 \longrightarrow \Ine(V, X) \longrightarrow \Dec(V, X) \overset{\gamma}{\longrightarrow} \Bir(X) \longrightarrow 1
\end{align}
exact, i.e., is $\gamma$ surjective?
\end{que}
Note that, in general, this sequence is not exact, i.e., $\gamma$ is not surjective (see Remark \ref{k3}).
When the sequence \eqref{se} is exact, the group $\Ine(V, X)$ measures how many ways one can extend $\Bir(X)$ to the birational automorphisms of the ambient space $V$.

Our main result is following theorem, answering a question asked by Ludmil Katzarkov:
\begin{thm}\label{intro}
Let $X_{n} \subset P(n+1)$ be a smooth hypersurface of multidegree $(2, 2, \ldots, 2)$ and $n \geq 3$. Then:
\begin{itemize}
\item[(1)] $\gamma \cc \Dec(P(n+1), X_{n}) \to \Bir(X_{n})$ is surjective, in particular Question $\ref{qu}$ is affirmative for $X_{n}$.
\item[(2)] If, in addition, $X_{n}$ is generic, there are $n+1$ elements $\rho_{i}$ $(1 \leq i \leq n+1)$ of $\Ine(P(n+1), X_{n})$ such that
\[
\langle \rho_{1}, \rho_{2}, \ldots , \rho_{n+1} \rangle \simeq \underbrace{\Z * \Z * \cdots * \Z}_{n+1} \subset \Ine(P(n+1), X_{n}).
\]
In particular, $\Ine(P(n+1), X_{n})$ is an infinite non-commutative group.
\end{itemize}
\end{thm}

Our proof of Theorem \ref{intro} is based on an explicit computation of elementary flavour.

We also consider another type of Calabi-Yau manifolds, namely smooth hypersurfaces of degree $n+2$ in $\PP^{n+1}$ and obtain the following result:

\begin{thm}\label{intro2}
Suppose $n \geq 3$. Let $M_{n} = (n+2) \subset \PP^{n+1}$ be a smooth hypersurface of degree $n+2$.
Then Question $\ref{qu}$ is also affirmative for $M_{n}$.
More precisely:
\begin{itemize}
\item[(1)] $\Dec(\PP^{n+1}, M_{n}) = \{ f \in \PGL(n+2, \C) = \Aut(\PP^{n+1})\ |\ f(M_{n}) = M_{n}\}$.
\item[(2)] $\Ine(\PP^{n+1}, M_{n}) = \{\id_{\PP^{n+1}}\}$, and $\gamma \cc \Dec(\PP^{n+1}, M_{n}) \overset{\simeq}{\longrightarrow} \Bir(M_{n}) = \Aut(M_{n})$.
\end{itemize}
\end{thm}

It is interesting that the inertia groups of $X_{n} \subset P(n+1) = (\PP^{1})^{n+1}$ and $M_{n} \subset \PP^{n+1}$ have completely different structures though both $X_{n}$ and $M_{n}$ are Calabi-Yau hypersurfaces in rational Fano manifolds.

\begin{rem}\label{k3}
There is a smooth quartic $K3$ surface $M_{2} \subset \PP^{3}$ such that $\gamma$ is not surjective (see \cite[Theorem 1.2 (2)]{og13}).
In particular, Theorem \ref{intro2} is not true for $n = 2$.
\end{rem}


\section{Preliminaries}

In this section, we prepare some definitions and properties of birational geometry and introduce the Cremona group.

\subsection{Divisors and singularities}
Let $X$ be a projective variety. A \textit{prime divisor} on $X$ is an irreducible subvariety of codimension one, and a \textit{divisor} (resp. \textit{$\Q$-divisor} or \textit{$\R$-divisor}) on $X$ is a formal linear combination $D = \sum d_{i}D_{i}$ of prime divisors where $d_{i} \in \Z$ (resp. $\Q$ or $\R$).
A divisor $D$ is called \textit{effective} if $d_{i}$ $\geq$ 0 for every $i$ and denote $D \geq 0$.
The closed set $\bigcup_{i}D_{i}$ of the union of prime divisors is called the \textit{support} of $D$ and denote Supp$(D)$. A $\Q$-divisor $D$ is called \textit{$\Q$-Cartier} if, for some $0 \neq m \in \Z$, $mD$ is a Cartier divisor (i.e. a divisor whose divisorial sheaf $\OO_{X}(mD)$ is an invertible sheaf), and $X$ is called $\Q$-\textit{factorial} if every divisor is $\Q$-Cartier.

Note that, since the regular local ring is the unique factorization domain, every divisor automatically becomes the Cartier divisor on the smooth variety.

Let $f \cc X \dashrightarrow Y$ be a birational map between normal projective varieties, $D$ a prime divisor, and $U$ the domain of definition of $f$; that is, the maximal subset of $X$ such that there exists a morphism $f \cc U \to Y$.
Then $\codim(X\setminus U) \geq 2$ and $D \cap U \neq \emptyset$, the image $(f|_{U})(D \cap U)$ is a locally closed subvariety of $Y$.
If the closure of that image is a prime divisor of $Y$, we call it the \textit{strict transform} of $D$ (also called the \textit{proper transform} or \textit{birational transform}) and denote $f_{*}D$.
We define $f_{*}D = 0$ if the codimension of the image $(f|_{U})(D \cap U)$ is $\geq$ 2 in $Y$.

We can also define the strict transform $f_{*}Z$ for subvariety $Z$ of large codimension; if $Z \cap U \neq \emptyset$ and dimension of the image $(f|_{U})(Z \cap U)$ is equal to $\dim Z$, then we define $f_{*}Z$ as the closure of that image, otherwise $f_{*}Z$ = 0.

Let $(X, D)$ is a \textit{log pair} which is a pair of a normal projective variety $X$ and a $\R$-divisor $D \geq 0$. For a log pair $(X, D)$, it is more natural to consider a \textit{log canonical divisor} $K_{X} + D$ instead of a canonical divisor $K_{X}$.

A projective birational morphism $g \cc Y \to X$ is a \textit{log resolution} of the pair $(X, D)$ if $Y$ is smooth, $\Exc(g)$ is a divisor, and $g_{*}^{-1}(D) \cup \Exc(g)$ has simple normal crossing support (i.e. each components is a smooth divisor and all components meet transversely) where $\Exc(g)$ is an exceptional set of $g$, and a divisor $over$ $X$ is a divisor $E$ on some smooth variety $Y$ endowed with a proper birational morphism $g \cc Y \to X$.

If we write
\[
K_{Y} + \Gamma + \sum E_{i} = g^{*}(K_{X}+D) + \sum a_{E_{i}}(X, D)E_{i},
\]
where $\Gamma$ is the strict transform of $D$ and $E_{i}$ runs through all prime exceptional divisors, then the numbers $a_{E_{i}}(X, D)$ is called the \textit{discrepancies of $(X, D)$ along $E_{i}$}. 
The \textit{discrepancy of} $(X, D)$ is given by
\[
\discrep(X, D) \ceq \inf\{ a_{E_{i}}(X, D)\ |\ E_{i} \text{ is a prime exceptional divisor over } X\}.
\]
The discrepancy $a_{E_{i}}(X, D)$ along $E_{i}$ is independent of the choice of birational maps $g$ and only depends on $E_{i}$.

Let us denote $\discrep(X, D) = a_{E}$.
A pair $(X, D)$ is \textit{log canonical} (resp. \textit{Kawamata log terminal} ($klt$)) if $a_{E} \geq 0$ (resp. $a_{E} > 0$).
A pair $(X, D)$ is \textit{canonical} (resp. \textit{terminal}) if $a_{E} \geq 1$ (resp. $a_{E} > 1$).

\subsection{Cremona groups}
Let $n$ be a positive integer.
The \textit{Cremona group} $\Cr(n)$ is the group of automorphisms of $\C(X_{1}, \ldots, X_{n})$, the $\C$-algebra of rational functions in $n$ independent variables.

Given $n$ rational functions $F_{i} \in \C(X_{1}, \ldots, X_{n})$, $1 \leq i \leq n$, there is a unique endomorphism of this algebra maps $X_{i}$ onto $F_{i}$ and this is an automorphism if and only if the rational transformation $f$ defined by $f(X_{1}, \ldots, X_{n}) = (F_{1}, \ldots, F_{n})$ is a birational transformation of the affine space $\A^{n}$.
Compactifying $\A^{n}$, we get
\[
\Cr(n) = \Bir(\A^{n}) = \Bir(\PP^{n})
\]
where Bir$(X)$ denotes the group of all birational transformations of $X$.

In the end of this section, we define two subgroups in $\Cr(n)$ introduced by Gizatullin \cite{gi94}.

\begin{dfn}
Let $V$ be an $(n+1)$-dimensional smooth projective rational manifold and $X \subset V$ a projective variety.
The \textit{decomposition group} of $X$ is the group
\[
\Dec(V, X) \ceq \{f \in \Bir(V)\ |\ f(X) =X \text{ and } f|_{X} \in \Bir(X) \}.
\]
The \textit{inertia group} of $X$ is the group
\[
\Ine(V, X) \ceq \{f \in \Dec(V, X)\ |\ f|_{X} = \id_{X}\}.
\]
\end{dfn}

The decomposition group is also denoted by Bir$(V, X)$.
By the definition, the correspondence
\[
\gamma \cc \Dec(V, X) \ni f \mapsto f|_{X} \in \Bir(X)
\]
defines the exact sequence:
\begin{align}\label{seq}
1 \longrightarrow \Ine(V, X) = \ker \gamma \longrightarrow \Dec(V, X) \overset{\gamma}{\longrightarrow} \Bir(X).
\end{align}

So, it is natural to consider the following question (which is same as Question \ref{qu}) asked by Ludmil Katzarkov:
\begin{que}\label{qexact}
Is the sequence
\begin{align}\label{exact}
1 \longrightarrow \Ine(V, X) \longrightarrow \Dec(V, X) \overset{\gamma}{\longrightarrow} \Bir(X) \longrightarrow 1
\end{align}
exact, i.e., is $\gamma$ surjective?
\end{que}

\begin{rem}
In general, the above sequence \eqref{exact} is not exact, i.e., $\gamma$ is not surjective.
In fact, there is a smooth quartic $K3$ surface $M_{2} \subset \PP^{3}$ such that $\gamma$ is not surjective (\cite[Theorem 1.2 (2)]{og13}).
\end{rem}


\section{Calabi-Yau hypersurface in $\PP^{n+1}$}\label{cyn}
Our goal, in this section, is to prove Theorem \ref{intro2} (i.e. Theorem \ref{ta}).
Before that, we introduce the result of Takahashi \cite{ta98}.

\begin{dfn}
Let $X$ be a normal $\Q$-factorial projective variety.
The 1\textit{-cycle} is a formal linear combination $C = \sum a_{i}C_{i}$ of proper curves $C_{i} \subset X$ which are irreducible and reduced.
By the theorem of the base of N\'eron-Severi (see \cite{kl66}), the whole numerical equivalent class of 1-cycle with real coefficients becomes the finite dimensional $\R$-vector space and denotes $N_{1}(X)$.
The dimension of $N_{1}(X)$ or its dual $N^{1}(X)$ with respect to the intersection form is called the \textit{Picard number} and denote $\rho(X)$.
\end{dfn}

\begin{thm}$($\cite[Theorem 2.3]{ta98}$)$\label{tak}
Let $X$ be a Fano manifold $($i.e. a manifold whose anti-canonical divisor $-K_{X}$ is ample,$)$ with $\dim X \geq 3$ and $\rho(X) = 1$, $S \in |-K_{X}|$ a smooth hypersurface with $\Pic(X) \to \Pic(S)$ surjective.
Let $\Phi \cc X \dashrightarrow X'$ be a birational map to a $\Q$-factorial terminal variety $X'$ with $\rho(X') = 1$ which is not an isomorphism, and $S' = \Phi_{*}S$.
Then $K_{X'} + S'$ is ample.
\end{thm}

This theorem is proved by using the \textit{Noether-Fano inequality} which is one of the most important tools in birational geometry, which gives a precise bound on the singularities of indeterminacies of a birational map and some conditions when it becomes isomorphism.

This inequality is essentially due to \cite{im71}, and Corti proved the general case of an arbitrary Mori fiber space of dimension three \cite{co95}.
It was extended in all dimensions in \cite{ta95}, \cite{bm97}, \cite{is01}, and \cite{df02}, (see also \cite{ma02}).
In particular, a log generalized version obtained independently in \cite{bm97}, \cite{ta95} is used for the proof of Theorem \ref{tak}.

After that, we consider $n$-dimensional \textit{Calabi-Yau manifold} $X$ in this paper.
It is a projective manifold which is simply connected,
\[
H^{0}(X, \Omega_{X}^{i}) = 0\ \ \ (0<i<\dim X = n),\ \ \textrm{and \ } H^{0}(X, \Omega_{X}^{n}) = \C \omega_{X},
\]
where $\omega_{X}$ is a nowhere vanishing holomorphic $n$-form.

The following theorem is a consequence of the Theorem \ref{tak}, which is same as Theorem \ref{intro2}.
This provides an example of the Calabi-Yau hypersurface $M_{n}$ whose inertia group consists of only identity transformation. 

\begin{thm}\label{ta}
Suppose $n \geq 3$. Let $M_{n} = (n+2) \subset \PP^{n+1}$ be a smooth hypersurface of degree $n+2$.
Then $M_{n}$ is a Calabi-Yau manifold of dimension $n$ and Question $\ref{qexact}$ is affirmative for $M_{n}$.
More precisely:
\begin{itemize}
\item[(1)] $\Dec(\PP^{n+1}, M_{n}) = \{ f \in \PGL(n+2, \C) = \Aut(\PP^{n+1})\ |\ f(M_{n}) = M_{n}\}$.
\item[(2)] $\Ine(\PP^{n+1}, M_{n}) = \{\id_{\PP^{n+1}}\}$, and $\gamma \cc \Dec(\PP^{n+1}, M_{n}) \overset{\simeq}{\longrightarrow} \Bir(M_{n}) = \Aut(M_{n})$.
\end{itemize}
\end{thm}

\begin{proof}
By Lefschetz hyperplane section theorem for $n \geq 3$, $\pi_{1}(M_{n}) \simeq \pi_{1}(\PP^{n+1}) = \{\id\}$, $\Pic(M_{n}) = \Z h$ where $h$ is the hyperplane class.
By the adjunction formula,
\[
K_{M_{n}} = (K_{\PP^{n+1}} + M_{n})|_{M_{n}} = -(n+2)h + (n+2)h = 0
\]
in Pic$(M_{n})$.

By the exact sequence
\[
0 \longrightarrow \OO_{\PP^{n+1}}(-(n+2)) \longrightarrow \OO_{\PP^{n+1}} \longrightarrow \OO_{M_{n}} \longrightarrow 0
\]
and
\[
h^{k}(\OO_{\PP^{n+1}}(-(n+2))) = 0\ \ \text{for}\ \ 1 \leq k \leq n,
\]
\[
H^{k}(\OO_{M_{n}}) \simeq H^{k}(\OO_{\PP^{n+1}}) = 0\ \ \text{for}\ \ 1 \leq k \leq n-1.
\]
Hence $H^{0}(\Omega^{k}_{M_{n}}) = 0$ for $1 \leq k \leq n-1$ by the Hodge symmetry.
Hence $M_{n}$ is a Calabi-Yau manifold of dimension $n$.

By $\Pic(M_{n}) = \Z h$, there is no small projective contraction of $M_{n}$, in particular, $M_{n}$ has no flop.
Thus by Kawamata \cite{ka08}, we get $\Bir(M_{n}) = \Aut(M_{n})$, and $g^{*}h = h$ for $g \in \Aut(M_{n}) = \Bir(M_{n})$.

So we have $g = \tilde{g}|_{M_{n}}$ for some $\tilde{g} \in \PGL(n+1, \C)$.
Assume that $f \in \Dec(\PP^{n+1}, M_{n})$.
Then $f_{*}(M_{n}) = M_{n}$ and $K_{\PP^{n+1}} + M_{n} = 0$.
Thus by Theorem \ref{tak}, $f \in \Aut(\PP^{n+1}) = \PGL(n+2, \C)$.
This proves (1) and the surjectivity of $\gamma$.

Let $f|_{M_{n}} = \id_{M_{n}}$ for $f \in \Dec(\PP^{n+1}, M_{n})$.
Since $f \in \PGL(n+1, \C)$ by (1) and $M_{n}$ generates $\PP^{n+1}$, i.e., the projective hull of $M_{n}$ is $\PP^{n+1}$, it follows that $f = \id_{\PP^{n+1}}$ if $f|_{M_{n}} = \id_{M_{n}}$.
Hence $\Ine(\PP^{n+1}, M_{n}) = \{\id_{\PP^{n+1}}\}$, i.e., $\gamma$ is injective.
So, $\gamma \cc \Dec(\PP^{n+1}, M_{n}) \overset{\simeq}{\longrightarrow} \Bir(M_{n}) = \Aut(M_{n})$.
\end{proof}


\section{Calabi-Yau hypersurface in $(\PP^{1})^{n+1}$}\label{cy1n}
As in above section, the Calabi-Yau hypersurface $M_{n}$ of $\PP^{n+1}$ with $n \geq 3$ has only identical transformation as the element of its inertia group.
However, there exist some Calabi-Yau hypersurfaces in the product of $\PP^{1}$ which does not satisfy this property; as result (Theorem \ref{main}) shows.

To simplify, we denote
\begin{align*}
P(n+1) &\ceq (\PP^{1})^{n+1} = \PP^{1}_{1} \times \PP^{1}_{2} \times \cdots \times \PP^{1}_{n+1},\\
P(n+1)_{i} &\ceq \PP^{1}_{1} \times \cdots \times \PP^{1}_{i-1} \times \PP^{1}_{i+1} \times \cdots \times \PP^{1}_{n+1} \simeq P(n),
\end{align*}
and
\begin{align*}
p^{i} \cc P(n+1) &\to \PP^{1}_{i} \simeq \PP^{1},\\
p_{i} \cc P(n+1) &\to P(n+1)_{i}
\end{align*}
as the natural projection.
Let $H_{i}$ be the divisor class of $(p^{i})^{*}(\OO_{\PP^{1}}(1))$, then $P(n+1)$ is a Fano manifold of dimension $n+1$ and its canonical divisor has the form $\displaystyle{-K_{P(n+1)} = \sum^{n+1}_{i=1}2H_{i}}$.
Therefore, by the adjunction formula, the generic hypersurface $X_{n} \subset P(n+1)$ has trivial canonical divisor if and only if it has multidegree $(2, 2, \ldots, 2)$.
More strongly, for $n \geq 3$, $X_{n} = (2, 2, \ldots, 2)$ becomes a Calabi-Yau manifold of dimension $n$ and, for $n=2$, a $K3$ surface (i.e. 2-dimensional Calabi-Yau manifold).
This is shown by the same method as in the proof of Theorem \ref{ta}.

From now, $X_{n}$ is a generic hypersurface of $P(n+1)$ of multidegree $(2, 2, \ldots , 2)$ with $n \geq 3$.
Let us write $P(n+1) = \PP^{1}_{i} \times P(n+1)_{i}$.
Let $[x_{i1} : x_{i2}]$ be the homogeneous coordinates of $\PP^{1}_{i}$.
Hereafter, we consider the affine locus and denote by $\displaystyle x_{i} = \frac{x_{i2}}{x_{i1}}$ the affine coordinates of $\PP^{1}_{i}$ and by ${\bf z}_{i}$ that of $P(n+1)_{i}$.
When we pay attention to $x_{i}$, $X_{n}$ can be written by following equation
\begin{align}\label{xn}
X_{n} = \{ F_{i,0}({\bf z}_{i})x_{i}^{2} + F_{i,1}({\bf z}_{i})x_{i} + F_{i,2}({\bf z}_{i}) = 0 \}
\end{align}
where each $F_{i,j}({\bf z}_{i})$ $(j = 0, 1, 2)$ is a quadratic polynomial of ${\bf z}_{i}$.
Now, we consider the two involutions of $P(n+1)$:
\begin{align}
\tau_{i} \cc (x_{i}, {\bf z}_{i}) &\to \left(-x_{i}- \frac{F_{i,1}({\bf z}_{i})}{F_{i,0}({\bf z}_{i})}, {\bf z}_{i} \right)\label{tau}\\
\sigma_{i} \cc (x_{i}, {\bf z}_{i}) &\to \left(\frac{F_{i,2}({\bf z}_{i})}{x_{i} \cdot F_{i,0}({\bf z}_{i})}, {\bf z}_{i} \right).\label{sigma}
\end{align}
Then $\tau_{i}|_{X_{n}} = \sigma_{i}|_{X_{n}} = \iota_{i}$ by definition of $\iota_{i}$ (cf. Theorem \ref{iota}).

We get two birational automorphisms of $X_{n}$
\begin{align*}
\rho_{i} = \sigma_{i} \circ \tau_{i} \cc (x_{i}, {\bf z}_{i}) &\to \left( \frac{F_{i,2}({\bf z}_{i})}{-x_{i} \cdot F_{i,0}({\bf z}_{i}) - F_{i,1}({\bf z}_{i})}, \  {\bf z}_{i} \right)\\
\rho'_{i} = \tau_{i} \circ \sigma_{i} \cc (x_{i}, {\bf z}_{i}) &\to \left( -\frac{x_{i} \cdot F_{i,1}({\bf z}_{i}) + F_{i,2}({\bf z}_{i})}{x_{i}\cdot F_{i,0}({\bf z}_{i})}, \  {\bf z}_{i} \right).
\end{align*}
Obviously, both $\rho_{i}$ and $\rho'_{i}$ are in Ine$(P(n+1), X_{n})$, map points not in $X_{n}$ to other points also not in $X_{n}$, and $\rho_{i}^{-1} = \rho'_{i}$ by $\tau_{i}^{2} = \sigma_{i}^{2} = \id_{P(n+1)}$.

\begin{prp}\label{order}
Each $\rho_{i}$ has infinite order.
\end{prp}

\begin{proof}
By the definiton of $\rho_{i}$ and $\rho'_{i} = \rho_{i}^{-1}$, it suffices to show	
\begin{align*}
{\begin{pmatrix}
0 & F_{i,2}\\
-F_{i,0} & -F_{i,1}
\end{pmatrix}}^{k}
\neq \alpha I
\end{align*}
for any $k \in \Z \setminus \{0\}$ where $I$ is an identity matrix and $\alpha \in \C^{\times}$.
Their eigenvalues are
\[
\frac{-F_{i,1} \pm \sqrt{F_{i,1}^{2} - 4F_{i,0}F_{i,2}}}{2}.
\]
Here $F_{i,1}^{2} - 4F_{i,0}F_{i,2} \neq 0$ as $X_{n}$ is general (for all $i$).

If $\begin{pmatrix}
0 & F_{i,2}\\
-F_{i,0} & -F_{i,1}
\end{pmatrix}^{k}
= \alpha I$
for some $k \in \Z \setminus \{0\}$ and $\alpha \in \C^{\times}$, then
\[
\left(\frac{-F_{i,1} + \sqrt{F_{i,1}^{2} - 4F_{i,0}F_{i,2}}}{-F_{i,1} - \sqrt{F_{i,1}^{2} - 4F_{i,0}F_{i,2}}}\right)^{k} = 1,
\]
a contradiction to the assumption that $X_{n}$ is generic.
\end{proof}
We also remark that Proposition \ref{order} is also implicitly proved in Theorem \ref{main}.

Our main result is the following (which is same as Theorem \ref{intro}):

\begin{thm}\label{main}
Let $X_{n} \subset P(n+1)$ be a smooth hypersurface of multidegree $(2, 2, \ldots, 2)$ and $n \geq 3$. Then:
\begin{itemize}
\item[(1)] $\gamma \cc \Dec(P(n+1), X_{n}) \to \Bir(X_{n})$ is surjective, in particular Question $\ref{qexact}$ is affirmative for $X_{n}$.
\item[(2)] If, in addition, $X_{n}$ is generic, $n+1$ elements $\rho_{i} \in \Ine(P(n+1), X_{n})$ $(1 \leq i \leq n+1)$ satisfy
\[
\langle \rho_{1}, \rho_{2}, \ldots , \rho_{n+1} \rangle \simeq \underbrace{\Z * \Z * \cdots * \Z}_{n+1} \subset \Ine(P(n+1), X_{n}).
\]
In particular, $\Ine(P(n+1), X_{n})$ is an infinite non-commutative group.
\end{itemize}
\end{thm}

Let $\Ind(\rho)$ be the union of the indeterminacy loci of each $\rho_{i}$ and $\rho^{-1}_{i}$; that is, $\displaystyle \Ind(\rho) = \bigcup_{i=1}^{n+1}\big(\Ind(\rho_{i}) \cup \Ind(\rho^{-1}_{i})\big)$ where $\Ind(\rho_{i})$ is the indeterminacy locus of $\rho_{i}$.
Clearly, $\Ind(\rho)$ has codimension $\geq 2$ in $P(n+1)$.

\begin{proof}
Let us show Theorem \ref{main} (1).
Suppose $X_{n}$ is generic.
For a general point $x \in P(n+1)_{i}$, the set $p_{i}^{-1}(x)$ consists of two points.
When we put these two points $y$ and $y'$, then the correspondence $y \leftrightarrow y'$ defines a natural birational involutions of $X_{n}$, and this is the involution $\iota_{j}$.
Then, by Cantat-Oguiso's result \cite[Theorem 3.3 (4)]{co11}, $\Bir(X_{n})$ $(n\geq 3)$ coincides with the group $\langle \iota_{1}, \iota_{2}, \ldots , \iota_{n+1} \rangle \simeq  \underbrace{\Z/2\Z * \Z/2\Z * \cdots * \Z/2\Z}_{n+1}$.

Two involutions $\tau_{j}$ and $\sigma_{j}$ of $X_{n}$ which we construct in \eqref{tau} and \eqref{sigma} are the extensions of the covering involutions $\iota_{j}$.
Hence, $\tau_{j}|_{X_{n}} = \sigma_{j}|_{X_{n}} = \iota_{j}$.
Thus $\gamma$ is surjective.
Since automorphisms of $X_{n}$ come from that of total space $P(n+1)$, it holds the case that $X_{n}$ is not generic.
This completes the proof of Theorem \ref{main} (1).

Then, we show Theorem \ref{main} (2).
By Proposition \ref{order}, order of each $\rho_{i}$ is infinite.
Thus it is sufficient to show that there is no non-trivial relation between its $n + 1$ elements $\rho_{i}$.
We show by arguing by contradiction.

Suppose to the contrary that there is a non-trivial relation between $n+1$ elements $\rho_{i}$, that is, there exists some positive integer $N$ such that
\begin{align}\label{rho}
\rho_{i_{1}}^{n_{1}} \circ \rho_{i_{2}}^{n_{2}} \circ \cdots \circ \rho_{i_{l}}^{n_{l}} = \id_{P(n+1)}
\end{align}
where $l$ is a positive integer, $n_{k} \in \Z\setminus\{0\}$ $(1\leq k \leq l)$, and each $\rho_{i_{k}}$ denotes one of the $n + 1$ elements $\rho_{i}$ $(1 \leq i \leq n+1)$ and satisfies $\rho_{i_{k}} \neq \rho_{i_{k+1}}$ $(0 \leq k \leq l-1)$.
Put $N = |n_{1}| + \cdots + |n_{l}|$.

In the affine coordinates $(x_{i_{1}}, {\bf z}_{i_{1}})$ where $x_{i_{1}}$ is the affine coordinates of $i_{1}$-th factor $\PP^{1}_{i_{1}}$, we can choose two distinct points $(\alpha_{1}, {\bf z}_{i_{1}})$ and $(\alpha_{2}, {\bf z}_{i_{1}})$, $\alpha_{1} \neq \alpha_{2}$, which are not included in both $X_{n}$ and $\Ind(\rho)$.

By a suitable projective linear coordinate change of $\PP^{1}_{i_{1}}$, we can set $\alpha_{1} = 0$ and $\alpha_{2} = \infty$.
When we pay attention to the $i_{1}$-th element $x_{i_{1}}$ of the new coordinates, we put same letters $F_{i_{1},j}({\bf z}_{i_{1}})$ for the definitional equation of $X_{n}$, that is, $X_{n}$ can be written by
\[
X_{n} = \{ F_{i_{1},0}({\bf z}_{i_{1}})x_{i_{1}}^{2} + F_{i_{1},1}({\bf z}_{i_{1}})x_{i_{1}} + F_{i_{1},2}({\bf z}_{i_{1}}) = 0 \}.
\]
Here the two points $(0, {\bf z}_{i_{1}})$ and $(\infty, {\bf z}_{i_{1}})$ not included in $X_{n} \cup \Ind(\rho)$.
From the assumption, both two equalities hold:
\begin{numcases}
{}
\rho_{i_{1}}^{n_{1}} \circ \cdots \circ \rho_{i_{l}}^{n_{l}}(0, {\bf z}_{i_{1}}) = (0, {\bf z}_{i_{1}}) & \\
\rho_{i_{1}}^{n_{1}} \circ \cdots \circ \rho_{i_{l}}^{n_{l}}(\infty, {\bf z}_{i_{1}}) = (\infty, {\bf z}_{i_{1}}).\label{infty}
\end{numcases}

We proceed by dividing into the following two cases.

{\flushleft
(i). The case where $n_{1} > 0$.
Write $\rho_{i_{1}} \circ \rho_{i_{1}}^{n_{1}-1} \circ \rho_{i_{2}}^{n_{2}} \circ  \cdots \circ \rho_{i_{l}}^{n_{l}} = \id_{P(n+1)}$.
}

Let us denote $\rho_{i_{1}}^{n_{1}-1} \circ \cdots \circ \rho_{i_{l}}^{n_{l}}(0, {\bf z}_{i_{1}}) = (p, {\bf z}_{i_{1}}')$, then, by the definition of $\rho_{i_{1}}$, it maps $p$ to $0$.
That is, the equation $F_{i_{1},2}({\bf z}'_{i_{1}}) = 0$ is satisfied.
On the other hand, the intersection of $X_{n}$ and the hyperplane $(x_{i_{1}}=0)$ is written by
\[
X_{n} \cap (x_{i_{1}}=0) = \{F_{i_{1},2}({\bf z}_{i_{1}}) = 0\}.
\]
This implies $(0, {\bf z}'_{i_{1}}) = \rho_{i_{1}}(p, {\bf z}'_{i_{1}}) = (0, {\bf z}_{i_{1}})$ is a point on $X_{n}$, a contradiction to the fact that $(0, {\bf z}_{i_{1}}) \notin X_{n}$.

{\flushleft
(ii). The case where $n_{1} < 0$.
Write $\rho^{-1}_{i_{1}} \circ \rho_{i_{1}}^{n_{1}+1} \circ \rho_{i_{2}}^{n_{2}} \circ  \cdots \circ \rho_{i_{l}}^{n_{l}} = \id_{P(n+1)}$.
}

By using the assumption \eqref{infty}, we lead the contradiction by the same way as in (i).
Precisely, we argue as follows.

Let us write $\displaystyle x_{i_{1}} = \frac{1}{y_{i_{1}}}$, then $(x_{i_{1}} = \infty, {\bf z}_{i_{1}}) = (y_{i_{1}} = 0, {\bf z}_{i_{1}})$ and $X_{n}$ and $\rho^{-1}_{i_{1}}$ can be written by
\[
X_{n} \ceq \{F_{i_{1},0}({\bf z}_{i_{1}}) + F_{i_{1},1}({\bf z}_{i_{1}})y_{i_{1}} + F_{i_{1},2}({\bf z}_{i_{1}})y_{i_{1}}^{2} = 0\},
\]
\[
\rho^{-1}_{i_{1}} \cc (y_{i_{1}}, {\bf z}_{i_{1}}) \to \left(\ -\frac{F_{i_{1},0}({\bf z}_{i_{1}})}{F_{i_{1},1}({\bf z}_{i_{1}}) + y_{i_{1}}\cdot F_{i_{1},2}({\bf z}_{i_{1}})},\ {\bf z}_{i_{1}} \right).
\]

Let us denote $ \rho_{i_{1}}^{n_{1}+1} \circ \rho_{i_{2}}^{n_{2}} \circ  \cdots \circ \rho_{i_{l}}^{n_{l}} (y_{i_{1}} =0, {\bf z}_{i_{1}}) = (y_{i_{1}} = q, {\bf z}_{i_{1}}'')$, then $\rho^{-1}_{i_{1}}$ maps $q$ to $0$.
That is, the equation $F_{i_{1},0}({\bf z}''_{i_{1}}) = 0$ is satisfied, but the intersection of $X_{n}$ and the hyperplane $(y_{i_{1}} = 0)$ is written by
\[
X_{n}\cap (y_{i_{1}} = 0) = \{F_{i_{1},0}({\bf z}_{i_{1}}) = 0\}.
\]
This implies $(y_{i_{1}}=0, {\bf z}''_{i_{1}}) = \rho^{-1}_{i_{1}}(y_{i_{1}} = q, {\bf z}_{i_{1}}'') = (x_{i_{1}}=\infty, {\bf z}_{i_{1}})$ is a point on $X_{n}$; that is, $(x_{i_{1}}=\infty, {\bf z}_{i_{1}}) \in X_{n} \cap (x_{i_{1}}=\infty)$.
This is contradiction.

\par

From (i) and (ii), we can conclude that there does not exist such $N$.
This completes the proof of Theorem \ref{main} (2).
\end{proof}

Note that, for the cases $n = 2$ and $1$, Theorem \ref{main} (2) also holds though (1) does not hold.

\vspace{1cm}

{\bf Acknowledgements: } The authors would like to express their sincere gratitude to their supervisor Professor Keiji Oguiso who suggested this subject and has given much encouragement and invaluable and helpful advices.

\vspace{0.5cm}

\end{document}